\documentclass[12pt]{amsart}

\usepackage{epsf,times}
\usepackage{amssymb}
\usepackage{amsmath}
\usepackage{amsthm}
\usepackage{amsbsy}
\usepackage{bm}
\usepackage{graphicx}
\usepackage[usenames,dvipsnames]{color}
\usepackage[all]{xy}

\theoremstyle{plain}
\newtheorem{theorem}{Theorem}[section]
\newtheorem{lemma}[theorem]{Lemma}
\newtheorem{corollary}[theorem]{Corollary}

\newtheorem*{quotthm}{Theorem}

\theoremstyle{definition}
\newtheorem{definition}[theorem]{Definition}
\newtheorem{example}[theorem]{Example}

\theoremstyle{remark}
\newtheorem{remark}[theorem]{Remark}
{\it}{\rm}

\newcommand{\R}{\mathbb{R}}

\newcommand{\del}{\partial}

\begin{document}
\title{Taut foliations in surface bundles with multiple boundary components}

\date{September 15, 2013}

\author[Kalelkar]{Tejas Kalelkar}
\address{Department of Mathematics, Washington University in St. Louis, St. Louis, MO 63130}
\email{tejas@math.wustl.edu}

\author[Roberts]{Rachel Roberts}
\address{Department of Mathematics, Washington University in St. Louis, St. Louis, MO 63130}
\email{roberts@math.wustl.edu}



\begin{abstract}
Let $M$ be a fibered  3-manifold with multiple boundary components. We show
that the fiber structure of $M$ transforms to closely related transversely oriented taut foliations
realizing all rational multislopes in some open neighborhood of the multislope of the fiber.
Each such foliation extends to a taut foliation in the closed 3-manifold obtained by Dehn filling along its boundary multislope.
The existence of these foliations implies that  certain contact structures are
weakly symplectically fillable.
\end{abstract}

\maketitle

\section{Introduction}

Any closed, orientable 3-manifold can be realized by Dehn filling a 3-manifold  which is fibered over $S^1$
\cite{Alex,Myers}.  In other words, any closed oriented 3-manifold can be realized by Dehn filling a 3-manifold
$M_0$, where 
$M_0$ has the form of a mapping torus
$$M_0 = S\times [0,1]/\sim,$$
where $S$ is  a compact orientable surface with nonempty boundary and $\sim$ is an equivalence relation given by $(x,1)\sim (h(x),0)$
for some orientation preserving homeomorphism $h:S\to S$  which fixes the components of $\partial S$ setwise. Although we shall not appeal to this fact in this paper, it is interesting to note that it is possible to assume that $h$ is pseudo-Anosov \cite{CH2}
and hence $M_0$ is hyperbolic \cite{Thurston}.  It is also possible to assume that $S$ has positive genus. Any nonorientable closed 3-manifold admits a double cover of this form.

Taut codimension one foliations are topological objects which have proved very useful in the study of 3-manifolds.   The problem of determining when a 3-manifold contains a taut foliation appears to be a very difficult one. A complete classification exists for Seifert fibered manifolds \cite{Brittenham,EHN,JN,Naimi} but relatively little is known for the case of hyperbolic 3-manifolds. There are many partial results demonstrating existence (see, for example, \cite{DR,Gab,Gab2,Gab3,Li0,LR,R,Ro,R2}) and partial results demonstrating nonexistence \cite{Jun,KM,KMOS,RSS}.  
In this paper, we investigate the existence of taut codimension one foliations in closed orientable 3-manifolds by first constructing taut codimension one foliations in corresponding mapping tori $M_0$. In contrast with the work in \cite{Ro,R2}, we consider the case that the boundary of $M_0$ is not connected.
We obtain the following results. Precise definitions will follow in Section~\ref{prelims}.

\begin{theorem}\label{Main}
Given an orientable, fibered compact 3-manifold, a fibration with fiber surface of positive genus can be modified to yield transversely oriented taut foliations realizing a neighborhood of  rational boundary multislopes about the boundary multislope of the fibration.
\end{theorem}

As an immediate corollary for closed manifolds we therefore have:

\begin{corollary}\label{Maincor}
Let $M = \widehat{M_0}(r^j)$ be the closed manifold obtained from $M_0$ by Dehn filling $M_0$  along the multicurve with rational multislope $(r^j)_{j=1}^k$. When $(r^j)$ is sufficiently close to the multislope of the fibration, $M$  admits a transversely oriented taut foliation.
\end{corollary}

Dehn filling $M_0$ along the slope of the fiber gives a mapping torus of a closed surface with the fibration as the obvious taut foliation. The above corollary shows that Dehn filling $M_0$ along slopes sufficiently close to the multislope of the fiber also gives a closed manifold with a taut foliation.

When the surgery coefficients $r^j$ are all meridians, the description of $M$ as a Dehn filling of $M_0$ gives an open book decomposition $(S,h)$ of $M$.  The foliations of Corollary~\ref{Maincor} can be approximated by a pair of contact structures, one positive and one negative, and both naturally related to the contact structure $\xi_{(S,h)}$ compatible with the open book decomposition   $(S,h)$  \cite{ET,KR2}. It follows that  
the contact structure $\xi_{(S,h)}$  is weakly symplectically fillable.

\begin{corollary}\label{contactcor}  Let $M$ have open book decomposition $(S,h)$. Then $M$ is obtained  
by Dehn filling $M_0$ along the multicurve with rational multislope $(r^j)_{j=1}^k$, where the $r^j$ are all meridians. When $(r^j)$ is sufficiently close to the multislope of the fibration, 
$\xi_{(S,h)}$  is weakly symplectically fillable and hence universally tight.
\end{corollary}
It is very natural to ask whether the qualifier `sufficiently close' can be made precise. 

Honda, Kazez, Matic\cite{HKMRV2} proved that when an open book with connected binding has monodromy with fractional Dehn twist coefficient $c$ at least one then it supports a contact structure which is close to a co-orientable taut foliation. Note that $c\ge 1$ is  sufficient but not always necessary to guarantee that $\xi_{(S,h)}$ is close to a co-orientable taut foliation. 

For an open book with multiple binding components, there is no such global lower  bound on the fractional Dehn twist coefficients  sufficient to guarantee that $\xi_{(S,h)}$ is close to a co-orientable taut foliation. This was shown by  Baldwin and Etnyre\cite{BE}, who constructed a sequence of open books with arbitrarily large fractional Dehn twist coefficients and disconnected binding that support contact structures which are not deformations of a taut foliation. So we can not expect to obtain a  neighborhood around the slope of the fiber which would satisfy our criteria of `sufficiently close' for every open book decomposition. At the end of the  paper, in Section \ref{Example}, we explicitly compute a neighborhood around the multislope of the fiber realizable by our construction for the Baldwin-Etnyre examples.

\section{Preliminaries}\label{prelims}

In this section we introduce basic definitions and fix conventions used in the rest of the paper. 

\subsection{Foliations}

Roughly speaking, a  codimension-1 foliation $\mathcal F$ of a 3-manifold $M$ is a disjoint union of injectively immersed surfaces such that $(M,\mathcal F)$ looks locally like $(\mathbb R^3,\R^2 \times \mathbb  R)$. 

\begin{definition}
Let $M$ be a closed $C^{\infty}$ 3-manifold and let $r$ be a non-negative integer or infinity. 
A \textit{$C^r$ codimension one foliation} $\mathcal F$ of (or in) $M$ is a union of disjoint  connected  surfaces $L_i$, called the \textit{leaves of $\mathcal F$}, in $M$ such that:
\begin{enumerate}
\item $\cup_i L_i = M$, and
\item  there exists a $C^r$ atlas $\mathcal A$ on $M$ which contains all $C^{\infty}$ charts and with respect to which $\mathcal F$ satisfies the following local product structure: 
\begin{itemize}
\item for every $p\in M$, there exists a coordinate chart $(U,(x,y,z))$ in $\mathcal A$ about $p$ such that $U\approx \mathbb R^3$ and the restriction of $\mathcal F$ to $U$ is the union of
planes given by $z = $ constant. 
\end{itemize}
\end{enumerate}
When $r=0$,  require also that the tangent plane field $T\mathcal F$ be $C^0$.
\end{definition}

A  \textit{taut} foliation \cite{Gab} is a codimension-1    foliation of a 3-manifold for which there exists a transverse simple closed curve that has nonempty intersection with each leaf of the foliation.  Although every 3-manifold contains a codimension-1 foliation \cite{Lic,No,Woo}, it is not true that every 3-manifold contains a codimension-1 taut foliation.
In fact, 
 the existence of a taut foliation in a closed orientable 3-manifold has important topological consequences for the manifold. For example, if $M$ is a closed, orientable 3-manifold that has a taut foliation with no sphere leaves then $M$ is covered by $\R^3$ \cite{Pa}, $M$ is irreducible \cite{Ros} and has an infinite fundamental  group \cite{Hae}.  In fact, its fundamental group acts nontrivially on interesting 1-dimensional objects (see, for example, \cite{T, CD} and \cite{Pa, RSS}).
Moreover,  taut foliations can be perturbed to weakly symplectically fillable contact structures \cite{ET} and hence can be used to obtain Heegaard-Floer information \cite{OzSz}.

\subsection{Multislopes}\label{multislope}

Let $F$ be a compact oriented surface of positive genus and with nonempty boundary consisting of $k$ components. Let $h$ be an orientation preserving homeomorphism of $F$ which fixes each boundary component pointwise. 
Let $M=F\times I/(x,1)\sim(h(x),0)$ and denote the $k$ (toral) boundary components of $\partial M$ by $T^1$, $T^2$, ..., $T^k$.

We use the given surface bundle structure on $M$ to fix a coordinate system  on each of the boundary tori as follows. 
(See, for example, Section 9.G of \cite{Rolf} for a definition and description of coordinate system.)
For each $j$ we choose as longitude $\lambda^j=\partial F\cap T^j$, with orientation inherited from the orientation of $F$.  
For each $j$, we then fix as meridian $\mu_j$ an oriented simple closed curve dual to $\lambda_j$. Although, as described in \cite{KR2,R2}, it is possible to use the 
homeomorphism $h$ to uniquely specify such simple closed curves $\mu_j$, we choose not to do so in this paper as all theorem statements are independent of the choice of meridional multislope.

We say a taut foliation $\mathcal{F}$ in $M$ {\it realizes boundary multislope}  $(m^j)_{j=1}^k$ if for each $j$, $1\leq j \leq k$, $\mathcal{F}\cap T^j$ is a foliation of $T^j$ of slope $m^j$ in the chosen coordinate system of $T^j$.

\subsection{Spines and Branched surfaces}

\begin{definition}
A {\it standard spine} \cite{C} is a space $X$ locally modeled on one of the spaces of Figure~\ref{spine}. The {\it critical locus} of $X$ is the 1-complex of points of $X$ where the spine is not locally a manifold.
\end{definition}

\begin{figure}
\centering
\def\svgwidth{0.7\columnwidth}
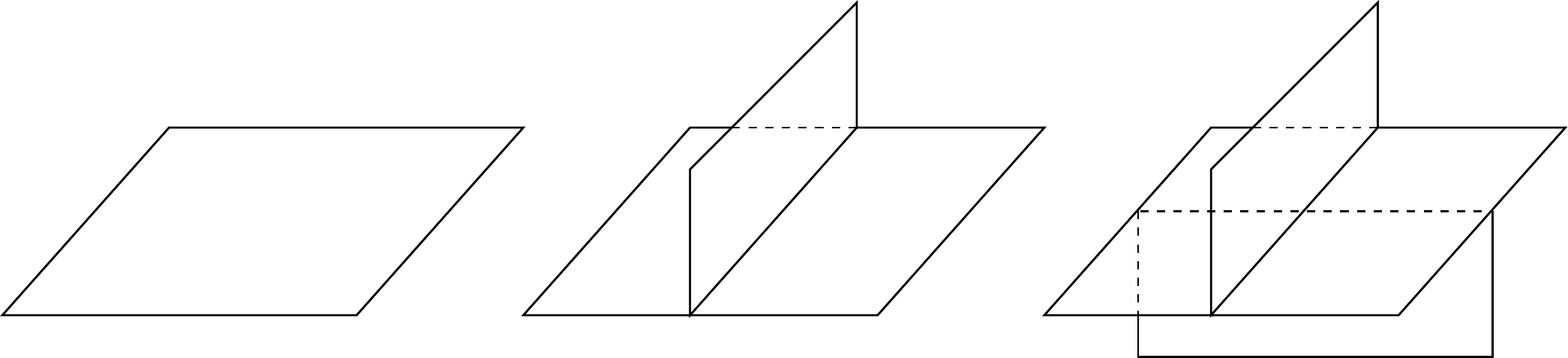
\caption{Local model of a standard spine}\label{spine}
\end{figure}

\begin{definition}
A {\it branched surface} (\cite{W}; see also \cite{O1,O2}) is a space $B$ locally modeled on the spaces of Figure~\ref{brsurf}. The {\it branching locus} $L$ of $B$ is the 1-complex of points of $B$ where $B$ is not locally a manifold.  The components of $B\setminus L$ are called the {\it sectors} of $B$. The points where $L$ is not locally a manifold are called {\it double points} of $L$.  
\end{definition}

\begin{figure}
\centering
\def\svgwidth{0.7\columnwidth}
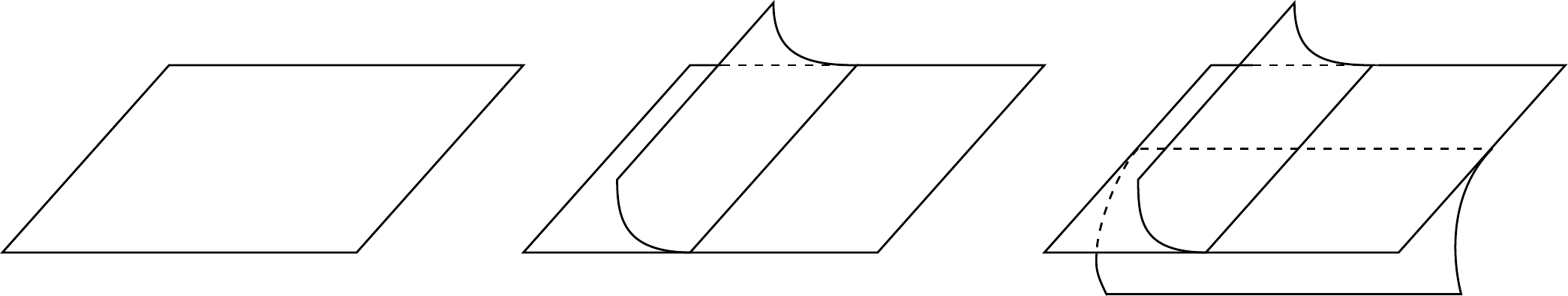
\caption{Local model of a branched surface}\label{brsurf}
\end{figure}

A standard spine $X$ together with an orientation in a neighborhood of the critical locus determines a branched surface $B$ in the sense illustrated in 
Figure~\ref{spinetobrsurf}.
\begin{figure}
\centering
\def\svgwidth{0.7\columnwidth}
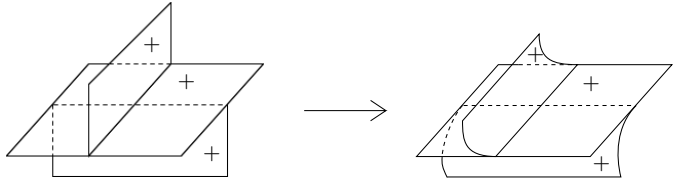
\caption{Oriented spine to oriented branched surface}\label{spinetobrsurf}
\end{figure}

\begin{example}
Let $F_0:= F\times\{ 0 \} $ be a fiber of $M=F\times I/(x,1)\sim(h(x),0)$. Let $\alpha_i$, $1\le i\le k$, be pairwise disjoint, properly embedded arcs in $F_0$, and set $D_i = \alpha_i \times I$ in $M$. Isotope the image arcs $h(\alpha_i)$ as necessary so that the intersection $(\cup \alpha_i)\cap (\cup_i h(\alpha_i))$ 
is transverse and minimal.  Assign an orientation to $F$ and to each $D_i$. Then $X=F_0\cup_i D_i$ is a transversely oriented spine. We will denote by $B = <F;\cup_i  D_i>$ the transversely oriented branched surface associated with $X=F_0\cup_i D_i$.

Similarly, $\langle \cup_i F_i; \cup_{i,j} D_i^j\rangle$ will denote the transversely oriented branched surface associated to the transversely oriented  standard spine 
$$X=F_0\cup F_{1}\cup ...\cup F_{n-1}\cup_{i,j} D_i^j $$ where $F_{i}=F\times \{i/n\}$ and $D_i^j=\alpha_i^j\times [\frac{i}{n},\frac{i+1}{n}]$ for some set of arcs $\alpha_i^j$ properly embedded in $F$ so that the intersection $(\cup_j \alpha_{i-1}^j)\cap (\cup_j \alpha_i^j)$ is transverse and minimal.
\end{example}
 
\begin{definition}
A lamination carried by a branched surface $B$ in $M$ is a closed subset $\lambda$ of an $I$-fibered regular neighborhood $N(B)$ of $B$, such that $\lambda$ is a disjoint union of injectively immersed 2-manifolds (called leaves) that intersect the $I$-fibers of $N(B)$ transversely.
\end{definition}

\subsection{Laminar branched surfaces} 

In \cite{Li0,Li}, Li introduces the  fundamental  notions of  sink disk and  half sink disk. 

\begin{definition} \cite{Li0,Li} Let $B$ be a branched surface in a 3-manifold $M$. Let $L$ be the branching locus of $B$ and let $X$ denote the union of double points of $L$.  Associate to each component of $L\setminus X$ a vector (in $B$) pointing in the direction of the cusp. A  {\it sink disk}  is a disk branch sector $D$  of 
$B$  for which the branch direction of each component of $(L\setminus X)\cap \overline{D}$ points into $D$  (as shown in Figure \ref{sinkdisk}).
A {\it half sink disk} is a sink disk which has nonempty intersection with $\partial M$.
\end{definition}

\begin{figure}
\centering
\def\svgwidth{0.5\columnwidth}
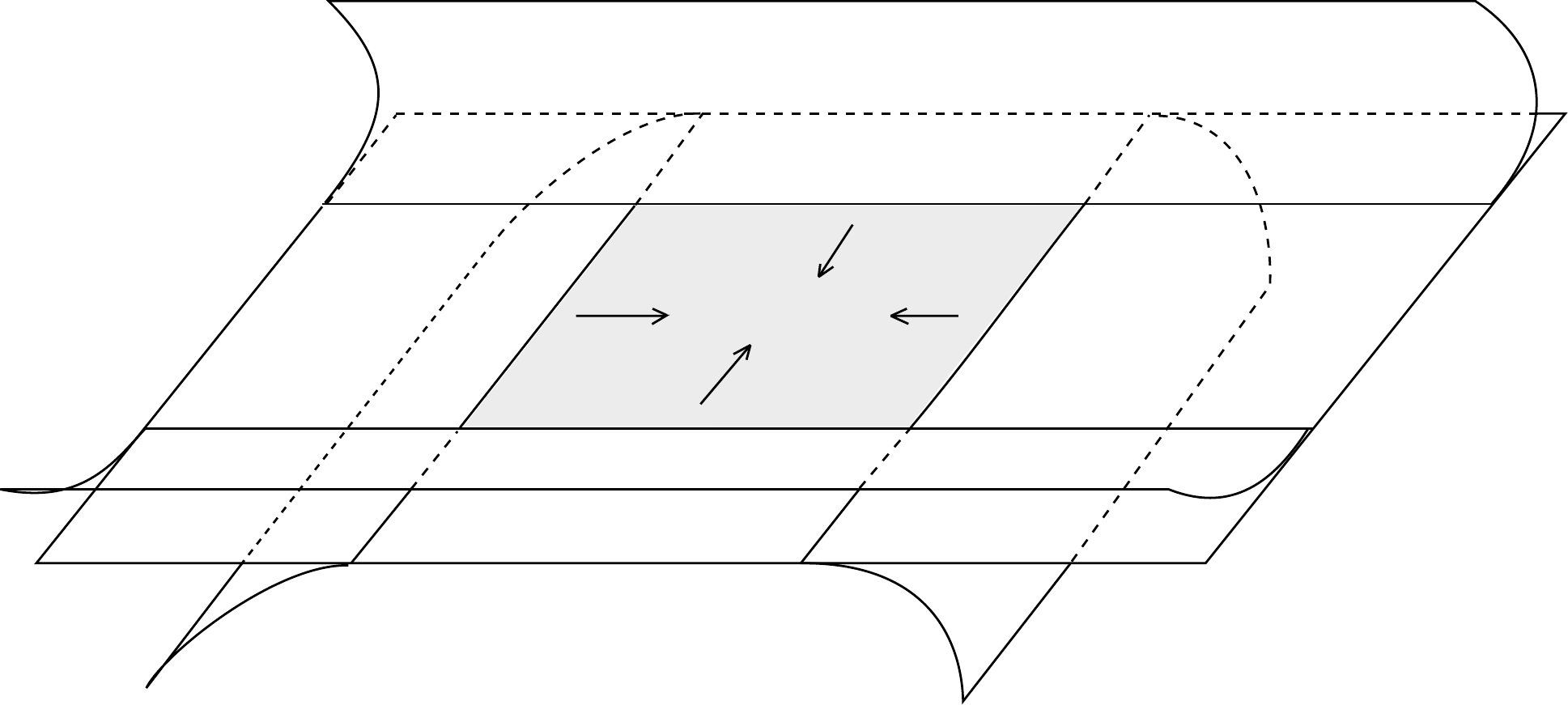
\caption{A sink disk}\label{sinkdisk}
\end{figure}

Sink disks and half sink disks play a key role in Li's notion  of  laminar branched surface.

\begin{definition} (Definition 1.3, \cite{Li0})
Let $D_1$ and $D_2$ be the two disk components of the horizontal boundary of a $D^2 \times I$ region in $M \setminus int(N(B))$. If the projection $\pi: N(B) \to B$ restricted to the interior of $D_1 \cup D_2$ is injective, i.e, the intersection of any $I$-fiber of $N(B)$ with $int(D_1) \cup int(D_2)$ is either empty or a single point, then we say that $\pi(D_1 \cup D_2)$ forms a \emph{trivial bubble} in $B$.
\end{definition}

\begin{definition} (Definition 1.4, \cite{Li0})
A branched surface $B$ in a closed 3-manifold $M$ is called a {\it laminar} branched surface if it satisfies the following conditions:
\begin{enumerate}
\item  $\partial_h N(B)$ is incompressible in $M\setminus int(N(B))$, no component of  $\partial_h N(B)$ is a sphere and $M\setminus B$ is irreducible.

\item There is no monogon in $M\setminus int(N(B))$; i.e., no disk $D\subset M\setminus int(N(B))$
with $\partial D = D\cap N(B) = \alpha\cup\beta$, where $\alpha\subset\partial_v N(B)$ is in an interval fiber of $\partial_v N(B)$ and $\beta\subset \partial_h N(B)$

\item There is no Reeb component; i.e., $B$ does not carry a torus that bounds a solid torus in $M$.
\item $B$ has no trivial bubbles.
\item $B$ has no sink disk or half sink disk.
\end{enumerate}
\end{definition}

Gabai and Oertel introduced essential branched surfaces in \cite{GO} and proved that any lamination fully carried by an essential branched surface is an essential lamination and conversely any essential lamination is fully carried by an essential branched surface. In practice, to check if a manifold has an essential lamination, the tricky part often is to verify that a candidate branched surface does in fact fully carry a lamination.  Li \cite{Li0} uses laminar branched surfaces to relax this requirement and prove the following:

\begin{theorem} (Theorem 1, \cite{Li0})
Suppose $M$ is a closed and orientable 3-manifold. Then
\begin{itemize}
\item[(a)] Every laminar branched surface in $M$ fully carries an essential lamination.
\item[(b)]Any essential lamination in $M$ that is not a lamination by planes is fully carried by a laminar branched surface.
\end{itemize}
\end{theorem}

In \cite{Li}, Li notices that if a branched surface has no half sink disk, then it can be arbitrarily split in a neighborhood of its boundary train track without introducing any sink disk (or half sink disk).
He is therefore able to conclude the following.

\begin{theorem} (Theorem 2.2, \cite{Li}) \label{tao}
Let $M$ be an irreducible and orientable 3-manifold whose boundary is a union of incompressible tori.
Suppose $B$ is a laminar branched surface and $\partial M\setminus \partial B$ is a union of bigons. 
Then, for any  multislope $(s_1,...,s_k) \in (\mathbb Q\cup \{\infty\})^k$ that can be realized by the train track $\partial B$, if $B$ does not carry a torus that bounds a solid torus in $\widehat{M}(s_1,...,s_k)$, then
$B$ fully carries a lamination $\lambda_{(s_1,...,s_k)}$ whose boundary consists of the multislope
$(s_1,...,s_k)$ and $\lambda_{(s_1,...,s_k)}$ can be extended to an essential lamination in $\widehat{M}(s_1,...,s_k)$.
\end{theorem}

We note that in \cite{Li}, Li states Theorem 2.2 only for the case that $\partial M$ is connected. However, as Li has observed and is easily seen, his proof extends immediately to the case that $\partial M$ consists of multiple toral boundary components.  Key is the fact that splitting $B$ open, to a branched surface $B'$ say, in a neighborhood of its boundary so that $\partial B'$ consists of multislopes $(s_1,...,s_k)$, does not  introduce sink disks. Therefore, capping
$B'$ off to $\widehat{B'}$ yields a laminar branched surface in $\widehat{M}(s_1,...,s_k)$.

\subsection{Good oriented sequence of arcs}
In this section we introduce some definitions that will be used in the rest of the paper.

\begin{definition}
Let $(\alpha^{1},...,\alpha^{k})$ be a tuple of pairwise disjoint simple arcs properly
embedded in $F$ with $\del\alpha^{j}\subset T^{j}$. Such a tuple
will be called \emph{parallel} if $F\setminus\{\alpha^{1},...,\alpha^{k}\}$
has $k$ components, $k-1$ of which are annuli $\{A^{j}\}$ with
$\del A^{j}\supset\{\alpha^{j},\alpha^{j+1}\}$ and one of which is a surface
$S$ of genus $g-1$ with $\del S\supset\{\alpha^{1},\alpha^{k}\}$.
Furthermore all $\alpha^{j}$ are oriented in parallel, i.e., the orientation
of $\del A^{j}$ agrees with $\{-\alpha^{j},\alpha^{j+1}\}$ and the orientation
of $\del S$ agrees with $\{-\alpha^{k},\alpha^{1}\}$.  Note that,
in particular, each $\alpha^{j}$ is non-separating. See Figure~\ref{Parallel tuple} for an example of a parallel tuple.
\end{definition}

\begin{figure}
\centering
\def\svgwidth{0.7\columnwidth}
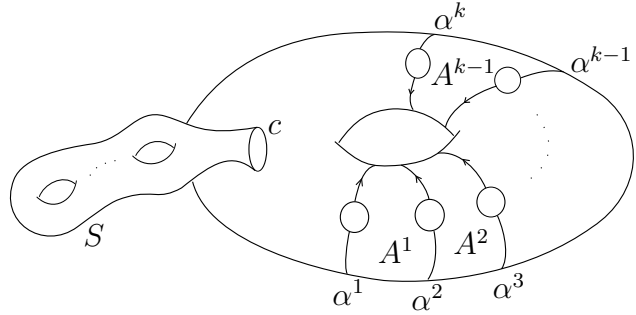
\caption{A Parallel tuple $(\alpha^i)$ on the surface $F$}\label{Parallel tuple}
\end{figure}

\begin{definition}
A pair of tuples $(\alpha^{i})_{i=1...k}$ and $(\beta^{j})_{j=1...k}$
will be called \emph{good} if both are parallel tuples and $\alpha^{i}$
and $\beta^{j}$ have exactly one  (interior) point of intersection when $i\neq j$ while $\alpha^{i}$
is disjoint from $\beta^{j}$ when $i=j$.

A sequence of parallel tuples $$\sigma=((\alpha_{0}^{1},\alpha_{0}^{2},...,\alpha_{0}^{k}),(\alpha_{1}^{1},\alpha_{1}^{2},...,\alpha_{1}^{k}),..., (\alpha_{n}^{1},\alpha_{n}^{2},...,\alpha_{n}^{k}))$$
also shortened to $$((\alpha_{0}^{j}),(\alpha_{1}^{j}),...,(\alpha_{n}^{j}))$$
or $$(\alpha_{0}^{j})\xrightarrow{{\sigma}}(\alpha_{n}^{j})$$ will
be called \emph{good} if for each fixed $j$, $1\leq j \leq k$,  the pair
$((\alpha_{i}^{j}),(\alpha_{i+1}^{j}))$ is good.
\end{definition}

\begin{figure}
\centering
\def\svgwidth{0.5\columnwidth}
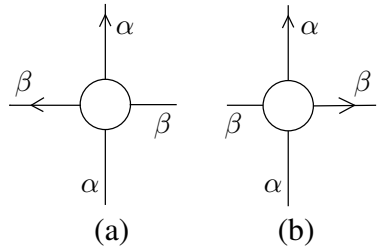
\caption{A pair of arcs $(\alpha,\beta)$ in position (a) is called negatively oriented, while a pair $(\alpha,\beta)$ in position (b) is called positively oriented}\label{Orientation}
\end{figure}

\begin{definition}
We say a good pair $((\alpha^{j}),(\beta^{j}))$ is \emph{positively
oriented} if for each $j\in\{1,...,k\}$ a neighborhood of the $j$-th
boundary component in $F$ is as shown in Figure~\ref{Orientation} (b). Similarly we say a good pair $((\alpha^{j}),(\beta^{j}))$ is \emph{negatively oriented} if for each
$j\in\{1,...,k\}$ a neighborhood of the $j$-th boundary component in $F$ is as shown in Fig~\ref{Orientation} (a). 

We say a good sequence $\sigma=((\alpha_{0}^{j}),(\alpha_{1}^{j}),...,(\alpha_{n}^{j}))$
is positively oriented if each pair $((\alpha_{i}^{j}),(\alpha_{i+1}^{j}))$
is positively oriented. Similarly we say $\sigma=((\alpha_{1}^{j}),(\alpha_{2}^{j}),...,(\alpha_{n}^{j}))$
is negatively oriented if each pair $((\alpha_{i}^{j}),(\alpha_{i+1}^{j}))$
is negatively oriented. We say the sequence $\sigma$ is oriented
if it is positively or negatively oriented. See Figure~\ref{spinewithgoodpair} for an example of a negatively oriented good pair in $F$.
\end{definition}

\subsection{Preferred generators}
Let 	$$\mathcal{H}_{g,k}=\{\eta_{1},\eta_{2},...,\eta_{2g-2+k,},\gamma_{12}, \gamma_{24}, \gamma_{46}, \gamma_{68},...,$$ $$\gamma_{2g-4,2g-2},\beta,\beta_{1},\beta_{2},...,\beta_{g-1},\delta_{1},\delta_{2},..., \delta_{k-1}\}$$ be the curves on  $F$ as shown in Figure~\ref{Gervais}. Then combining Proposition 1 and Theorem 1 of Gervais \cite{Ge} the mapping class group $MCG(F,\del F)$ of $F$ (fixing boundary) is generated by Dehn twists about curves in $\mathcal{H}_{g,k}$. 
\begin{quotthm}
[Gervais] The mapping class group $MCG(F, \del F)$ of $F$ is generated by Dehn twists about the curves in $\mathcal{H}_{g,k}$.
\end{quotthm}

\begin{figure}
\centering
\def\svgwidth{\columnwidth}
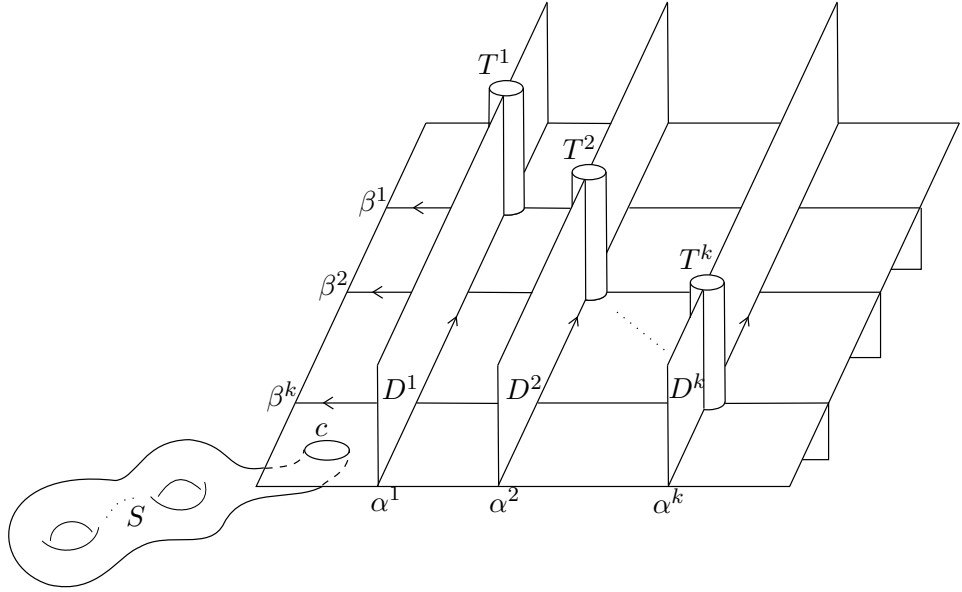
\caption{Neighborhood of $F$ with a good negatively oriented pair $((\alpha^j),(\beta^j))$ in the oriented spine $X$}\label{spinewithgoodpair}
\end{figure}

As Dehn twists about $\delta_{i}$ are isotopic to the identity via
an isotopy that does not fix the boundary, we have the following obvious corollary:
\begin{corollary}\label{Gervaisthm}
The mapping class group $MCG(F)$ of $F$ (not fixing the boundary pointwise) is generated by Dehn twists about the curves in $$\mathcal{H}_{g,k}'=\mathcal{H}_{g,k}\setminus\{\delta_{1},...,\delta_{k-1}\}$$
\end{corollary}

\section{Main Theorem}

\begin{definition}\label{orientation}
Let $(\alpha^{1},...,\alpha^{k})$ be a parallel tuple in $F$. Orient
$F$ so that the normal vector $\hat{n}$ induced by the orientation of
$M$ points in the  direction of increasing $t\in[0,1]$. Let $D^{j}=\alpha^{j}\times[0,1]$
in $M_{h}$ with the orientation induced by orientations of $\alpha^{j}$
and $F$; i.e., if $v^{j}$ is tangent to $\alpha^{j}$ then $(v^{j},\hat{n})$
gives the orientation of $D^{j}$.  Let $X=F\cup_{j}D^{j}$ be an oriented
standard spine and $B_{\alpha}=<F;\cup_jD^{j}>$ the transversely oriented branched  surface associated with
$X$. 
\end{definition}

\begin{figure}
\centering
\def\svgwidth{0.8\columnwidth}
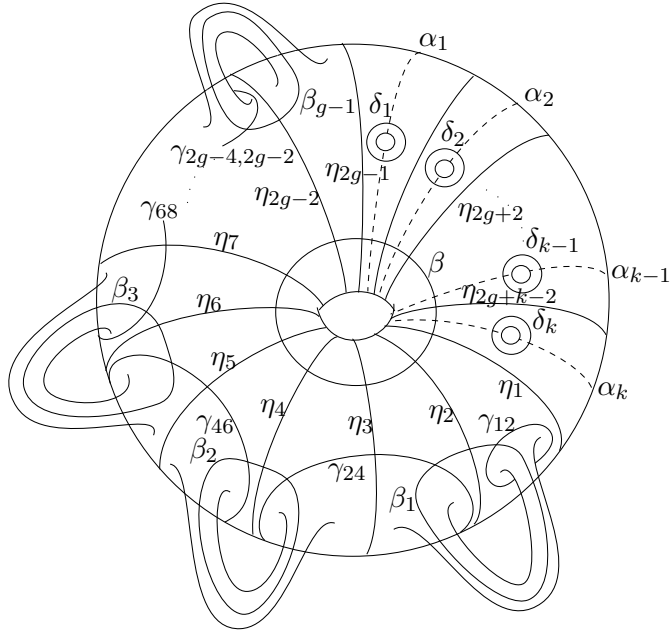
\caption{Generators of the Mapping Class Group}\label{Gervais}
\end{figure}

Notice that the multislope of the fibration is $\bar{0}$. In order to prove Theorem \ref{Main} we shall prove the following:

\begin{theorem}\label{detailed main} There is an open neighborhood $\mathcal U$ of $\bar{0}\in\mathbb R^k$ such that 
for each point $(m^{1},...,m^{k})\in \mathcal U\cap \mathbb Q^k$,
there exists a lamination carried by $B_{\alpha}$ with
boundary multislope $(m^{j})$. These laminations extend to taut
foliations which also intersect the boundary in foliations with multislope $(m^j)$.
\end{theorem}

This gives us the following corollary for closed manifolds.
\begin{corollary}
Let $\widehat{M}(r^j)$ denote the closed manifold obtained from $M$ by a Dehn filling along a multicurve with rational multislope $(r^j)_{j=1}^k$. For each  tuple $(r^j) \in \mathcal U\cap\mathbb Q^k$, the closed manifold $\widehat{M} (r^j)$ also has a transversely oriented  taut foliation.
\end{corollary}

The proof of Theorem \ref{detailed main} is outlined below with details worked out in the mentioned lemmas.
\begin{proof}
In Lemma \ref{lem:Existence of good sequence} we shall show that there is
a good positively oriented sequence  $(\alpha_{0}^{j})\to(h^{-1}(\alpha_{0}^{j}))$, or equivalently
from $(h(\alpha_{n}^{j}))\to(\alpha_{n}^{j})$. In Lemma \ref{lem:affine measure}
we shall show that whenever there exists such a positive sequence
there is a splitting of the branched surface $B_{\alpha}$  to a branched surface $B_{\sigma}$ that  is laminar and that therefore carries
laminations realizing every multislope in some open neighborhood of $\bar{0}\in\mathbb R^k$.
And finally in Lemma \ref{lem:foliation is taut} we show that these
laminations extend to taut foliations on all of $M$.
\end{proof}

\begin{lemma}
\label{lem:Existence of good sequence}Let $(\alpha^{j})$ be a parallel
tuple in $F$ and let $h\in Aut^{+}(F)$. Then there is a good
positively oriented sequence 
$(\alpha^{j})\xrightarrow{\sigma}(h(\alpha^{j}))$.
\end{lemma}

\begin{proof} 
By Corollary \ref{Gervaisthm} to Gervais' Theorem, 
$h\sim h_{m}h_{m-1}...h_{2}h_{1}$ for twists $h_{i}$
about curves in $\mathcal{H}'_{g,k}$. Set   $h'=h_{m}h_{m-1}...h_{2}h_{1}$, and notice that $M_{h}=M_{h'}$.

By changing the handle decomposition of $F$ as necessary, we may  assume 
that the parallel tuple $(\alpha^j)$ is as shown in Figure~\ref{Gervais}. Let
$b$ denote the Dehn twist about $\beta\in\mathcal{H}'_{g,k}$. 
Notice that any $h_i$ in the factorization of $h'$ is either $b$, $b^{-1}$, or a twist about a curve disjoint from all components of $\alpha^j$.  Thus $((\alpha^j),(h_i(\alpha^j))$ is either a good positive pair, a good negative pair,  or a pair of equal tuples.

Now if $((\alpha^{j}),(\beta^{j}))$ is a good pair then so is $((h_{i}(\alpha^{j})),(h_{i}(\beta^{j})))$; therefore each of the pairs \\
$((\alpha^{j}),(h_{m}(\alpha^{j}))$, $((h_{m}(\alpha^{j})),(h_{m}h_{m-1}(\alpha^{j})))$, $((h_{m}h_{m-1}(\alpha^{j})), (h_{m}h_{m-1} \\ h_{m-2}(\alpha^{j})))$, ..., $((h_{m}h_{m-1}...h_{2}(\alpha^{j})), (h_{m}h_{m-1}...h_{2}h_{1}(\alpha^{j})=h(\alpha^{j})))$\\
is either a good oriented pair or a pair of equal tuples.

If at least one of the $h_{i}$ is $b$ or $b^{-1}$ then ignoring the equal tuples, we get a good oriented sequence $((\alpha_{0}^{j}),(\alpha_{1}^{j}),...,(\alpha_{n-1}^{j}),(\alpha_{n}^{j})=h((\alpha_{0}^{j})))$ or $(\alpha^{j})\xrightarrow{\sigma}(h(\alpha^{j}))$ as required. The length of this sequence is equal to the number of times $h_{i}$ equals $b$
or $b^{-1}$, i.e, $n=n_{+}+n_{-}$, where $n_{+}$ is the sum of the positive powers of $b$ in this expression of $h'$
and $n_{-}$ is the magnitude of the sum of negative powers of $b$.

If none of the $h_{i}$ are Dehn twists about $\beta$ then $(\alpha^{j})=(h(\alpha^{j}))$.
In this case, $\sigma=((\alpha^{j}),(b(\alpha^{j})),(b^{-1}b(\alpha^{j})=(\alpha^{j})))$
is a good oriented sequence.

If $((\alpha^{j}),(\beta^{j}))$ is a positively oriented good pair
then $((\alpha^{j})$,$(-\beta^{j})$,$(-\alpha^{j})$,\\$(\beta^{j}))$ is a
negatively oriented good sequence. Performing $n_{-}$ such substitutions
we get a positively oriented good sequence $(\alpha^{j})\xrightarrow{\sigma}(h(\alpha^{j}))$.
\end{proof}

\begin{definition}\label{Associated branched surface}
Let $\sigma=(h(\alpha_{n}^{j})=\alpha_{0}^{j},\alpha_{1}^{j},....,\alpha_{n-1}^{j},\alpha_{n}^{j})$
be a good oriented sequence. Let $F_{i}=F\times\{\frac{i}{n}\}$ for
$0\leq i<n$ and let $D_{i}^{j}=\alpha_{i}^{j}\times[\frac{i}{n},\frac{i+1}{n}]$,
for $0\leq i<n$, in $M_{h}$. Let $X=(\cup_{i}F_{i})\cup(\cup_{i,j}D_{i}^{j})$
and orient $F_i$ and $D_i^j$ as in Definition \ref{orientation}. Define $B_{\sigma}=<\cup_i F_{i};\cup_{i,j}D_{i}^{j}>$ as the associated branched surface. Figure~\ref{spinewithgoodpair} shows the neighborhood of $F$ in $X$ while Figure~\ref{mainfig} shows a neighborhood of $F$ in the associated branched surface.
\end{definition}

\begin{figure}
\centering
\def\svgwidth{1\columnwidth}
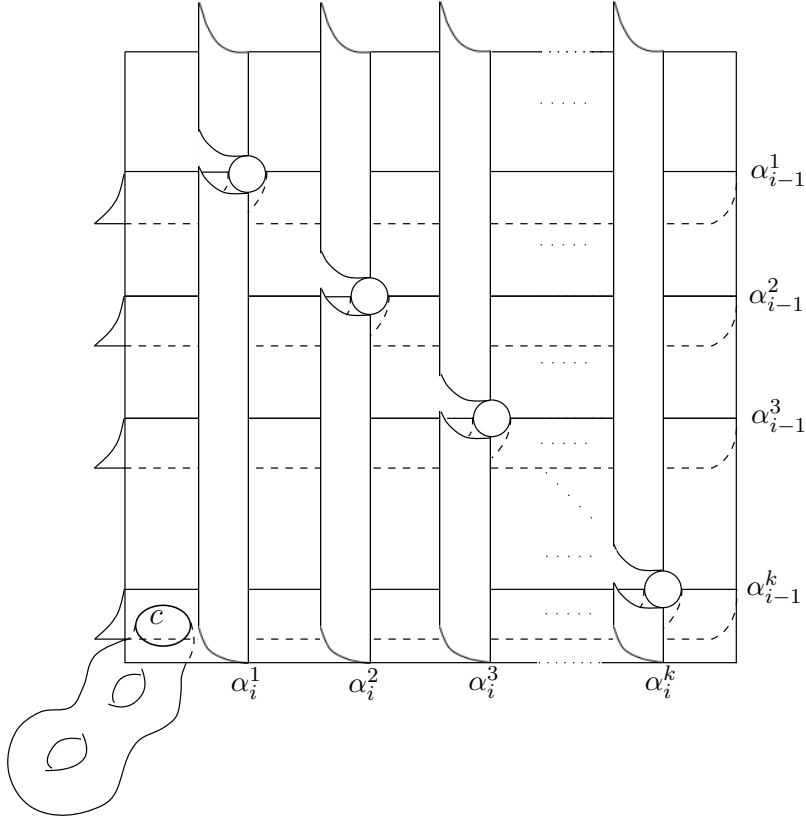
\caption{A neighborhood of one of the fibres in the branched surface $B$. The small circles along the diagonal   represent longitudes of the boundary tori. The vertical subarcs of the boundaries of the vertical disk sectors lie on these boundary tori. Compare with Figure~\ref{spinewithgoodpair}.}\label{mainfig}
\end{figure}

\begin{lemma}
\label{lem:affine measure}Let $\sigma=(h(\alpha_{n}^{j})=\alpha_{0}^{j},\alpha_{1}^{j},....,\alpha_{n-1}^{j},\alpha_{n}^{j})$
be a good oriented sequence in $F$ and $B_{\sigma}$ the associated
branched surface in $M_{h}$. Then $B_{\sigma}$ has no sink disk or half sink disk.
\end{lemma}

\begin{proof}
As the sequence $\sigma$ is good and oriented for each fixed $i$, the tuple of arcs $(\alpha_i^j)$ is parallel and $|\alpha_i^j \cap \alpha_{i-1}^k|=\delta_j^k$, so a neighborhood of $F_{i}$ in $B_{\sigma}$ is as shown in  Figure~\ref{mainfig}.
 
The sectors of $B_{\sigma}$ consist of disks $D_i^j=\alpha_i^j\times [\frac{i}{n},\frac{i+1}{n}]$ and components of $F_i\setminus \{ \alpha_i^j\cup\alpha_{i-1}^j\}_{j=1...k}$.   As $F_{i-1}$ and $F_i$ both have a co-orientation in the direction of increasing $t$ for $(x,t) \in M_h$, so for any orientation of $D_i^j$, $\del D_i^j$ is the union of  two arcs in $\del M_h$, together with one arc with the direction of the cusp pointing into the disk and one arc with the direction of the cusp pointing outwards. Similarly, as $\alpha_i^j$ and $\alpha_i^{j+1}$ are oriented in parallel, each disk component of $F_i \setminus \{\alpha_i^j, \alpha_{i-1}^j\}_{j=1...k}$ has a boundary arc with cusp direction pointing outwards. Therefore no branch sector in $B_\sigma$ is a sink disk or a half sink disk.
\end{proof}

\begin{remark}
Notice that $B_{\sigma}=<\cup_i F_{i};\cup_{i,j} D_{i}^{j}>$ is a splitting (see \cite{O2}) of  the original branched surface $B_{\alpha}=<F;\cup_j D^{j}>$ and, equivalently, $B_\sigma$ collapses to $B_{\alpha}$. So in particular, laminations carried by $B_\sigma$ are also carried by $B_\alpha$.

\end{remark}

\begin{figure}
\centering
\def\svgwidth{.5\columnwidth}
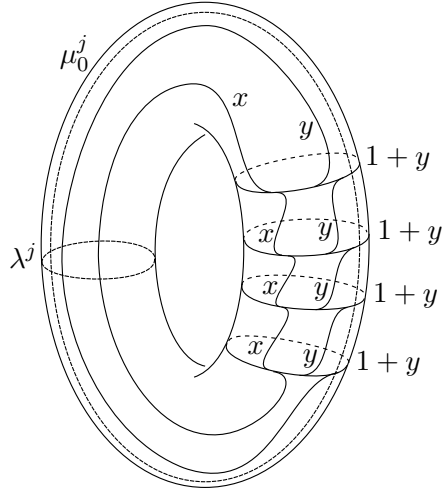
\caption{The weighted boundary train track when $n=4$ }\label{finaltt}
\end{figure}

Now consider the train tracks $\tau^j=B_\sigma\cap T^j$. Focus on one of the $\tau^j$. For some choice
of meridian $\mu_0^j$. Recall that we fixed a coordinate system $(\lambda^j,\mu^j)$ on $T^j$. For simplicity of exposition, we now make a second choice $\mu_0^j$ of meridian. This choice is dictated by the form of $\tau^j$; namely, we choose $\mu_0^j$ to be disjoint from the disks $D_i^j$ so that $\tau^j$  has the form shown in Figure~\ref{finaltt}. Notice that there is a change of coordinates homeomorphism taking slopes in terms to the coordinate system $(\lambda^j,\mu_0^j)$ to slopes in terms of the coordinate system $(\lambda^j,\mu^j)$. Since $\lambda^j$ is unchanged, this homeomorphism takes an open interval about $0$ to an open interval about $0$. Assign to $\tau^j$ the measure determined by weights $x,y$ shown in Figure~\ref{finaltt}. In terms of the coordinate system $(\lambda^j,\mu_0^j)$, $\tau^j$ carries all slopes realizable by $$\frac{x-y}{n(1+y)}$$ for some $x, y >0$.
Therefore, in terms of the coordinate system 
$(\lambda^j,\mu_0^j)$, $\tau^j$ carries all slopes in $(-\frac{1}{n},\infty)$. Converting to the coordinate system 
$(\lambda^j,\mu^j)$, $\tau^j$ carries all slopes in some open neighborhood of $0$. 
Repeat for all $j$. By Theorem~\ref{tao}, we see that the branched surface  $B_{\sigma}$ carries laminations 
$\lambda_{(\bar{x},\bar{y})}$ realizing multislopes  $(\frac{x_1-y_1}{n(1+y_1)},\frac{x_2-y_2}{n(1+y_2)},...,\frac{x_k-y_k}{n(1+y_k)})$ for any strictly positive values of $x_1,...,x_k,y_1,...,y_k$ and hence realizing all rational multislopes in some open neighborhood of $\bar{0}\in \mathbb R^k$.

\begin{lemma}
\label{lem:foliation is taut} Suppose the weights $\bar{x},\bar{y}$ are distinct and have strictly positive coordinates. Then
each  lamination $\lambda_{(\bar{x},\bar{y})}$, 
contains only noncompact leaves. Furthermore, each  lamination  $\lambda_{(\bar{x},\bar{y})}$
extends to a taut foliation $\mathcal{F}_{(\bar{x},\bar{y})}$, which realizes the same multislope.
\end{lemma}

\begin{proof}
 Suppose that  $\lambda_{(\bar{x},\bar{y})}$ contains a compact leaf $L$.  Such a leaf determines a transversely invariant measure on $B$ given by counting intersections with $L$.

Now focus on  any $i,j$, where $0\le i,j<n$. By considering, for example, a simple closed curve in $F_i$ parallel to the arc $\alpha_i^j$, we see that there is an oriented simple closed curve in $F_i$ which intersects the branching locus of $B_{\sigma}$ exactly $k$ times that has orientation consistent with the branched locus. Since this is true for all possible $i,j$, it follows that the only transversely invariant measure $B$ can support is the one with all weights on the branches $D_i^j$ necessarily $0$. But this means that $\lambda_{(\bar{x},\bar{y})}$ realizes multislope 
$\bar{0}$ and hence that $\bar{x}=\bar{y}$.

The complementary regions to the lamination $\lambda_{(\bar{x},\bar{y})}$ are product
regions. Filling these up with product fibrations, we get the required foliation $\mathcal{F}_{(\bar{x},\bar{y})}$,  which also has no compact leaves and is therefore taut.
\end{proof}

\section{Example}\label{Example}

 As discussed in the introduction, an open book with connected binding and monodromy with fractional Dehn twist coefficient more than one supports a contact structure which is the deformation of a co-orientable taut foliation \cite{HKMRV2}. However for open books with disconnected binding there is no such universal lower bound on the fractional Dehn twist coefficient. This was illustrated by Baldwin-Etnyre \cite{BE} who constructed a sequence of open books with arbitrarily large fractional Dehn twist coefficients and disconnected binding that support contact structures which are not deformations of a taut foliation. This shows, in particular, that there is no global neighborhood  about the multislope of the fiber of a surface bundle such that Dehn filling along rational slopes in that neighborhood produces closed manifolds with taut foliations. 

The notion of `sufficiently close' in Corollary \ref{Maincor} can however be bounded below for a given manifold. Deleting a neighbourhood of the binding in the Baldwin-Etnyre examples gives a surface bundle and using the techniques developed in the previous sections we now calculate a neighborhood  of multislopes realized by taut foliations, around the multislope of the fiber in this fibration. In particular, we  observe that this neighbourhood does not contain the meridional multislope. So Dehn filling along these slopes does not give a taut foliation of the sequence of Baldwin-Etnyre manifolds, as is to be expected.

The following is a description of the Baldwin-Etnyre examples\cite{BE}. Let $T$ denote the genus one surface with two boundary components, $B_1$ and $B_2$. Let $\psi$ be the diffeomorphisms of $T$ given by the product of Dehn twists,
$$\psi = D_a D_b^{-1} D_c D_d^{-1}$$
where $a,b,c$ and $d$  are the curves shown in Figure \ref{BaldwinEtnyre} (reproduced from Figure 1 of Baldwin-Etnyre\cite{BE}). Then $\psi$ is pseudo-Anosov by a well-known construction of Penner \cite{Pen}. We define $$\psi_{n,k_1,k_2}=D_{\delta_1}^{k_1} D_{\delta_2}^{k_2}\psi^n$$ where $\delta_1$ and $\delta_2$ are curves parallel to the boundary components $B_1$ and $B_2$ of $T$. 

\begin{figure}
\centering
\includegraphics[width=0.5\textwidth]{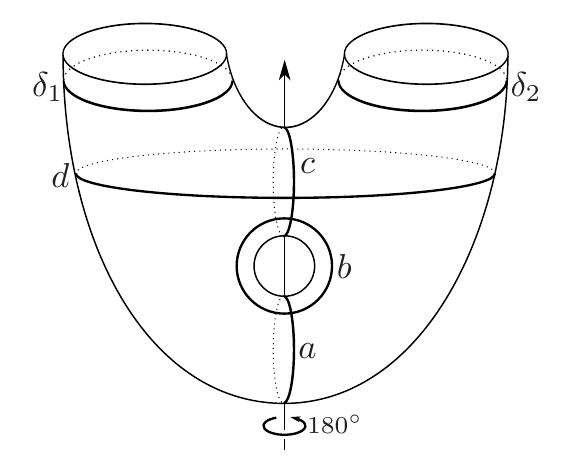}
\caption{The Baldwin-Etnyre examples}\label{BaldwinEtnyre}
\end{figure}

Let $M_{n,k_1, k_2}$ be the open book $(T, \psi_{n,k_1,k_2}$).  Let $N(B_{1})$, $N(B_{2})$ be regular neighbourhoods of $B_{1}$ and $B_{2}$ in $M_{n,k_{1},k_{2}}$ and let $M'_{n,k_{1},k_{2}}=M_{n,k_{1},k_{2}}\setminus(N(B_{1})\cup N(B_{2}))$.
Let $\lambda_{1}$, $\lambda_{2}$ be the closed curves in $T\cap\del M'_{n,k_{1},k_{2}}$ represented by $B_{1}$, $B_{2}$, with induced orientation. The monodromy
$\psi_{n,k_{1},k_{2}}$ is freely isotopic to the pseudo-Anosov map
$\psi^{n}$. Let $\mu_{1}$, $\mu_{2}$ be the suspension flow of
a point in $\lambda_{1}$ and $\lambda_{2}$ respectively under the
monodromy $\psi^{n}$. As $\psi^{n}$ is the identity on $\del T$,
  $\mu_i=p_{i}\times S^1$ in $\del M'_{n,k_{1},k_{2}}=(B_1 \times S^1)\cup(B_2 \times S^1)$ for $p_{i}\in\lambda_{i}$.

We use these pair of dual curves $(\lambda_{1},\mu_{1})$ and $(\lambda_{2},\mu_{2})$ as coordinates to calculate the slope of curves on the torii boundary
of $M'_{n,k_{1},k_{2}}$ as detailed in Subsection~\ref{multislope}.

If $D_{1}$ is the meridional disk of a regular neigbourhood $N(B_{1})$
of $B_{1}$ in $M_{n,k_{1,}k_{2}}$ then $\del D_{1}=\mu_{1}$. Similarly,
for $D_{2}$ a meridional disk of a regular neighbourhood of $B_{2}$
in $M_{n,k_{1},k_{2}}$, $\del D_{2}=\mu_{2}$.

In order to express the monodromy of the surface bundle in terms of
the Gervais generators we use the pseudo-Anosov monodromy
$\psi^n = \psi_{n,0,0}$ which is freely isotopic to $\psi_{n,k_{1},k_{2}}$, with the observation that Dehn filling $M'_{n,0,0}$ along slopes $-\frac{1}{k_{1}}$ and
$-\frac{1}{k_{2}}$ gives the manifold $M_{n,k_{1},k_{1}}$. So for
$M'_{n,0,0}$ we have $slope(\del D_{1})=-\frac{1}{k_{1}}$, $slope(\del D_{2})=-\frac{1}{k_{2}}$. 

As shown in Theorem 1.16 of Baldwin-Etnyre\cite{BE}, for any $N>0$ there exist $n,k_1 > N$ such that the corresponding open book in $M_{n,k_1,n}$ has a compatible contact structure that is not a deformation of the tangent bundle of a taut foliation. We shall now show that the slope $-\frac{1}{n}$ lies outside the interval of perturbation that gives slopes of taut foliations via our construction. And hence, the manifolds $M_{n,k_1,n}$ cannot be obtained by capping off the taut foliations realized by our interval of boundary slopes around the fibration.

To obtain the branched surface required in our construction in the previous sections  we need a good sequence of arcs $\alpha^{j}\to\phi^{-1}(\alpha^{j})$ where $\phi=\psi^{n}$, $j=1,2$. These arcs are used to construct product disks which we then smoothen along copies of the fiber surface to get the required branched surface. 

Following the method outlined in Lemma \ref{lem:Existence of good sequence} we need to express $\phi^{-1}$ in terms of the Gervais generators. The curves $a$,
$b$ and $c$ correspond to the generating curves $\eta_{1}$, $\beta$
and $\eta_{2}$ among the Gervais generators as can be seen in Figure \ref{Gervais}.
We now need to express the curve $d$ in terms of these generating curves.

\begin{definition}
Let $S_{g,n}$ be a surface of genus $g$ and $n$ boundary components. Consider a subsurface of $S_{g,n}$ homeomorphic to
$S_{1,3}$. Then for curves $\alpha_{i}$, $\beta$, $\gamma_{i}$
as shown in the Figure \ref{starrelation} (reproduced from Figure 2 of Gervais\cite{Ge}), the star-relation is $$(D_{\alpha_{1}}D_{\alpha_{2}}D_{\alpha_{3}}D_{\beta})^{3}=D_{\gamma_{1}}D_{\gamma_{2}}D_{\gamma_{3}}$$
where $D$ represents Dehn-twist along the corresponding curves.
\end{definition}

\begin{figure}
\centering
\def\svgwidth{0.4\columnwidth}
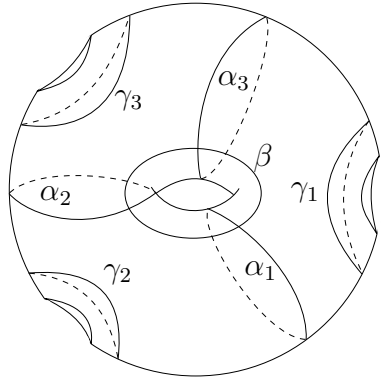
\caption{The Gervais star-relation}\label{starrelation}
\end{figure}

Let $S$ be the component of $T\setminus d$ which is homeomorphic
to a once-punctured torus. Let $\gamma_{1}=d$ and $\gamma_{2}$, $\gamma_{3}$
be curves bounding disjoint disks $D_1$ and $D_2$ in $S$ so that $S\setminus(D_1 \cup D_2)$
is homeomorphic to $S_{1,3}$. As $\gamma_{2},\gamma_{3}$ are trivial
in $T$, $\gamma_{1}=d$ and $\alpha_{1}=\alpha_{2}=\alpha_{3}=a$,
so the star relation reduces to $D_{d}=(D_{a}^{3}D_{b})^{3}$.

Hence, the monodromy $\psi$ in terms of the Gervais generators is
the word $$\psi = D_{a}D_{b}^{-1}D_{c}(D_{a}^{3}D_{b})^{-3}$$ which gives
us $$\psi^{-1}=D_{a}^{3}D_{b}D_{a}^{3}D_{b}D_{a}^{3}D_{b}D_{c}^{-1}D_{b}D_{a}^{-1}$$
Take arcs $\alpha_{1}$, $\alpha_{2}$ as shown in Figure \ref{Gervais}, where $a=\eta_{1}$, $b=\beta$ and
$c=\eta_{2}$. Then as $(\alpha_{j},D_{b}(\alpha_{j}))$ is a negatively
oriented pair and $\alpha_{j}=D_{a}(\alpha_{j})$, $\alpha_{j}=D_{c}(\alpha_{j})$
so we have a negatively oriented good sequence $(\alpha_{1},\alpha_{2})\to(\psi^{-1}(\alpha_{1}),\psi^{-1}(\alpha_{2}))$
obtained by taking the sequence of arcs $$\sigma=(\alpha_{j}, D_{a}^{3}D_{b}(\alpha_{j}), D_{a}^{3}D_{b}D_{a}^{3}D_{b}(\alpha_{j}), 
D_{a}^{3}D_{b}D_{a}^{3}D_{b}D_{a}^{3}D_{b}(\alpha_{j}),$$
$$D_{a}^{3}D_{b}D_{a}^{3}D_{b}D_{a}^{3}D_{b}D_{c}^{-1}D_{b}D_{a}^{-1}(\alpha_{j})=\psi^{-1}(\alpha_{j})) \mbox{ for $j=1,2$.} $$

Let $B_{\sigma}$ be the branched surface corresponding to this good
oriented sequence as in Definition \ref{Associated branched surface}. The weighted train track $\tau_{\sigma}=B_{\sigma}\cap\del M'_{n,0,0}$ on the boundary torii
is as shown in Figure \ref{finaltt}.

The slope of this measured boundary lamination is $\frac{x-y}{4(1+y)}$ so the interval of slopes that are realized by taut foliaions is $(-\frac{1}{4}, \infty)$. 

When the monodromy is $\psi^n$ (instead of $\psi$), by a similar argument, we get the slope of the measured lamination on the boundary as $\frac{x-y}{4n(1+y)}$ so that the interval of slopes realized by taut foliations is $(-\frac{1}{4n},\infty)$. And we observe that the point $(-\frac{1}{k_1}, -\frac{1}{n}) \notin (-\frac{1}{4n},\infty) \times (-\frac{1}{4n},\infty)$, i.e, the taut foliations from our construction cannot be capped off to give a taut foliation of the Baldwin-Etnyre examples.

\bibliographystyle{amsplain}

\begin{thebibliography}{10}


\bibitem{Alex} J. W. Alexander, \emph{Note on Riemann spaces}, Bull. Amer. Math. Soc., \textbf{26} (1920), 370--372.


\bibitem{BE}
J.\ Baldwin and J.\ Etnyre,
\emph{Admissible transverse surgery does not preserve tightness},
arXiv:1203.2993

\bibitem{Brittenham}  M.\ Brittenham, \emph{Essential laminations in Seifert-fibered spaces}, Topology \textbf{32(1)} (1993), 61--85.

\bibitem{CD}
D.\ Calegari and N.\ Dunfield,
\emph{Laminations and groups of homeomorphisms of the circle},
Invent. Math. \textbf{152} (2003), 149--207.

\bibitem{C} B.\ Casler, \emph{An imbedding theorem for connected 3-manifolds with boundary}, Proc. A.M.S. \textbf{16} (1965), 559--566.

\bibitem{CH2} V.\ Colin and K.\ Honda, \emph{Stabilizing the monodromy of an open book decomposition}, Geom. Dedicata \textbf{132} (2008), 95--103.

\bibitem{DR}
C. \  Delman and R.\ Roberts, 
\emph{Alternating knots satisfy Strong Property P}. 
Comment. Math. Helv. \textbf{74} (1999), no. 3, 376--397. 


\bibitem{ET}
Y.\ Eliashberg and W.\ P.\ Thurston,
\emph{Confoliations}, University Lecture Series \textbf{13}, American Mathematical Society (1998).

\bibitem{EHN} D.\ Eisenbud, U.\ Hirsch and W.\ D.\ Neuman, \emph{Transverse foliations of Seifert bundles and self-homeomorphisms of the circle}, Comment. Math. Helv. \textbf{56(4)} (1981), 638--660.



\bibitem{Gab} D.\ Gabai, \emph{Foliations and the topology of 3-manifolds},
J. Differential Geom. \textbf{18} (1983), no. 3, 445--503. 

\bibitem{Gab2} D.\ Gabai, \emph{Foliations and the topology of 3-manifolds. II}, J. Differential Geom. \textbf{26} (1987), no. 3, 461--478.

\bibitem{Gab3} D.\  Gabai, \emph{Foliations and the topology of 3-manifolds. III}, J. Differential Geom. \textbf{26} (1987), no. 3, 479--536.

\bibitem{GO}D.\ Gabai and U.\ Oertel, {\it Essential Laminations in 3-Manifolds}, Ann. Math. {\bf 130}, (1989), 41--73.

\bibitem{Ge} S.\  Gervais, \emph{A finite presentation of the mapping class group of a punctured surface}, Topology \textbf{40} (2001), no. 4, 703--725.



\bibitem{Hae} A. Haefliger, \emph{Varietes feuilletés}, Ann. Scuola Norm. Sup. Pisa (3) \textbf{16} (1962) 367--397 

\bibitem{HKMRV2} K. Honda, W.H. Kazez and G. Matic, \emph{Right-veering diffeomorphisms of compact surfaces with boundary II}, Geometry \& Topology, \textbf{12}, (2008),  2057--2094.

\bibitem{JN} M.\ Jankins and W.\ D.\ Neumann, \emph{Rotation numbers of products of circle homeomorphisms}, Math. Ann. \textbf{271(3)} (1985), 381--400.

\bibitem{Jun}  J.\ Jun, \emph{$(-2,3,7)$-pretzel knot and Reebless foliation}, arXiv:math/0303328v1.

\bibitem{KM} P.\ B.\ Kronheimer and T.\ Mrowka, \textit{Monopoles and three-manifolds}, Cambridge University Press, 2007.


\bibitem{KMOS}
P. Kronheimer, T. Mrowka, P. Ozsv\'ath, and Z. Szab\'o, 
\emph{Monopoles and lens space surgeries}, Annals of Math. \textbf{165} (2007), 457--546.

\bibitem{KR} W.\ H.\ Kazez and R.\ Roberts, {\it Fractional Dehn twists in knot theory and contact topology},   \texttt{ArXiv:1201.5290}. 

\bibitem{KR2} W.\ H.\ Kazez and R.\ Roberts, {\it Continuous confoliations}, preprint.

\bibitem{Li0} T.\ Li, \emph{Laminar branched surfaces in 3-manifolds}, Geom. Top. \textbf{6} (2002), 153--194.

\bibitem{Li}  T.\  Li, \emph{Boundary train tracks of laminar branched surfaces.}, Topology and geometry of manifolds (Athens, GA, 2001), 269–-285, Proc. Sympos. Pure Math., \textbf{71}, Amer. Math. Soc., Providence, RI, 2003.

\bibitem{LR} T.\ Li and R.\ Roberts, {\it Taut foliations in knot complements}, to appear in Pac. J. Math., \texttt{ArXiv:1211.3066}.

\bibitem{Lic} W.B. Lickorish, \emph{A foliation for 3-Manifolds}, 
Ann. of Math. (2) \textbf{82} 414--420 1965.


\bibitem{Myers} R.\ Myers, \emph{Open book decompositions of 3-manifolds}, Proc. A.M.S. \textbf{72(2)} (1978),  397--402. 

\bibitem{Naimi} R.\ Naimi, \emph{Foliations transverse to fibers of Seifert manifolds}, Commet. Math. Helv. \textbf{69(1)} (1994), 155--162.

\bibitem{No} S.P. Novikov, \emph{The topology of foliations.} (Russian), Trans. Moscow Math. Society \textbf{14} (1965),  248--278. 

\bibitem{O1} U.\ Oertel, \emph{Incompressible branched surfaces}, Invent. Math. \textbf{76} (1984), 385--410.

\bibitem{O2} U.\ Oertel, \emph{Measured laminations in 3-manifolds}, Trans, A.M.S. \textbf{305} (1988), 531--573.

\bibitem{OzSz}
P.\ Ozsv\'ath and Z.\ Szab\'o, 
\emph{Holomorphic disks and genus bounds}, Geom. Topol. \textbf{8} (2004), 311--334.

\bibitem{Pa} C.\ F.\ B.\ Palmeira,  \textit{Open manifolds foliated by planes}, Ann. Math. (2) \textbf{107} (1978), no. 1, 109--131.

\bibitem{Pen} R.\ Penner, \emph{A construction of pseudo-Anosov homeomorphisms.}, Trans. Amer. Math. Soc. \textbf{310} (1988), no. 1, 179--197. 


\bibitem{R}
R.\  Roberts,
\emph{Constructing taut foliations},
Comment. Math. Helv.  \textbf{70} (1995), 516--545. 

\bibitem{Ro} R.\ Roberts, \emph{Taut foliations in punctured surface bundles. I}, Proc. London Math. Soc. (3) \textbf{82} (2001), no. 3, 747--768.

\bibitem{R2}
R.\  Roberts,
\emph{Taut foliations in punctured surface bundles, II}. 
Proc. London Math. Soc. (3) \textbf{83} (2001), no. 2, 443--471.

\bibitem{RSS}
R.\ Roberts, J.\  Shareshian and M.\  Stein, 
\emph{Infinitely many hyperbolic 3-manifolds which contain no Reebless foliation}, J.A.M.S. \textbf{16(3)}, 639--379.


\bibitem{Rolf} 
D.\  Rolfsen, \emph{Knot and Links}, Mathematics Lecture Series, No. 7. Publish or Perish, Inc., Berkeley, Calif., 1976.

\bibitem{Ros}
H. Rosenberg, \emph{Foliations by planes}, Topology \textbf{6} (1967) 131--138

\bibitem{Thurston} W.\  P.\ Thurston \textit{On the geometry and dynamics of diffeomorphisms of surfaces}, Bull.\ Amer.\ Math. \ Soc.   {\bf 19}, (1988), pp 417--431.

\bibitem{Th} W.\ P.\ Thurston, \emph{A norm for the homology of 3-manifolds}, 
Mem. Amer. Math. Soc. 59 (1986), no. 339, i--vi and 99--130. 

\bibitem{T} W.\ P.\  Thurston, \emph{Three-manifolds, foliations and circles, II},
Unfinished manuscript, 1998.



\bibitem{W} R.\ Williams, \emph{Expanding attractors}, Inst. Hautes \'Etudes Sci. Publ. Math. \textbf{43} (1974), 169--203.

\bibitem{Woo} J.\ Wood, \emph{Foliations on 3-manifolds}, Ann. of Math. \textbf{89(2)} (1969), 336--358. 



\end{thebibliography}

\end{document}